\documentclass[11pt,a4paper]{amsart}

\usepackage[T1]{fontenc}
\usepackage[utf8]{inputenc}
\usepackage{lmodern}
\usepackage{microtype}
\usepackage{amsmath,amssymb,amsthm,mathtools}
\usepackage[margin=30mm]{geometry}
\usepackage{xcolor}
\usepackage[colorlinks=true,linkcolor=black!65!black,citecolor=black!65!black,urlcolor=black!65!black]{hyperref}

\theoremstyle{plain}
\newtheorem{theorem}{Theorem}[section]
\newtheorem{proposition}[theorem]{Proposition}
\newtheorem{lemma}[theorem]{Lemma}
\newtheorem{corollary}[theorem]{Corollary}
\theoremstyle{definition}
\newtheorem{definition}[theorem]{Definition}
\theoremstyle{remark}
\newtheorem{remark}[theorem]{Remark}

\newcommand{\CL}{\operatorname{CL}}
\newcommand{\W}{\tau_{W}}
\newcommand{\V}{\tau_{V}}
\title[Baire Wijsman hyperspaces]{On baireness of the Wijsman topology}
\author{Basile Beyer de Ryke}

\address{Independent researcher, 12 Avenue Hamoir, 1180 Uccle, Belgium}
\email{basile.beyer@gmail.com}
\date{August 2026}

\begin{document}

\begin{abstract}
We give negative answers to Questions~4.4 and~4.5 in Cao's survey on the
Wijsman topology. We construct a metrizable space $Y$ and a
compatible metric $d$ such that every finite power $Y^n$ is Baire, while
$(\CL(Y),\tau_W(d))$ is not Baire.  The construction realizes the Vietoris
hyperspace of a suitable metric space as an open subspace of a ball
hyperspace.

We also study the set $S_d$ of points having an open
$\sigma$-$d$-totally bounded neighborhood.  We prove that Baireness of
$(\CL(X),\tau_W(d))$ implies Baireness of $S_d$, and that, whenever $S_d$ is
dense in $X$,
$$
(\CL(X),\tau_W(d))\text{ is Baire}
\quad\Longleftrightarrow\quad
X\text{ is Baire}
\quad\Longleftrightarrow\quad
S_d\text{ is Baire}.
$$
As consequences, we obtain complete characterizations for locally
$\sigma$-$d$-totally bounded metric spaces, countable metric spaces, and
$\sigma$-compact metrizable spaces.  In our counterexample, $S_e$ coincides with a non-dense discrete component for every compatible metric $e$, showing that the density assumption is essential.
\end{abstract}

\maketitle

\section{Introduction}

Let $(X,d)$ be a metric space and let $\CL(X)$ denote the set of all
nonempty closed subsets of $X$.  The Wijsman topology $\tau_W(d)$ is defined as the weakest topology such that each function of the family
$$
A\longmapsto d(x,A)=\inf_{a\in A}d(x,a),\qquad x\in X.
$$
is continuous.

The study of Baireness in hyperspaces has an earlier counterpart for the
Vietoris topology.  Cao, Garc\'ia-Ferreira and Gutev proved that if the
Vietoris hyperspace of a space $X$ is Baire, then every finite power $X^n$ is
Baire \cite{CaoGarciaFerreiraGutev2007}; the proof of this result was
subsequently corrected in an erratum \cite{CaoGarciaFerreiraGutev2008}.
Cao and Tomita later showed that if $X$ is quasi-regular and $X^{\aleph_0}$ is
Baire, then its Vietoris hyperspace is Baire, while Baireness of all finite
powers is not sufficient \cite{CaoTomita2007}.  Related Vietoris Baireness
questions were also studied by Garc\'ia-Ferreira and Gutev through
web-adjacent spaces and a construction motivated by Krom
\cite{GarciaFerreiraGutev2009}.  These results provide the natural
finite-versus-countable-power background for the corresponding questions in
the Wijsman setting.

Against this background, Cao's Problem~4.1 asks for conditions on $(X,d)$
which characterize when $(\CL(X),\tau_W(d))$ is Baire \cite{Cao2017}.  Two
natural subquestions are particularly fundamental.  Question~4.4 asks whether
Baireness of every finite power $X^n$ is sufficient, while Question~4.5,
attributed there to Zsilinszky, asks whether Baireness of $X$ itself is
sufficient.

Our first theorem provides simultaneous negative answers to these two questions.

The proof starts from the Cao--Tomita example of a metrizable space $X$ whose
finite powers are Baire while $(\CL(X),\tau_V)$ is not Baire
\cite{CaoTomita2007}.  We attach one auxiliary point $p_F$ for each
$F\in\CL(X)$ and one additional point $q$.  The metric is chosen so that a
radius-one closed ball about $p_F$ has trace exactly $F$ on $X$, while a
radius-one ball about $q$ removes all auxiliary points simultaneously.  This
realizes the Vietoris hyperspace as an open subspace of a ball hyperspace.
Zsilinszky's equivalence between Baireness of the ball and Wijsman topologies
then transfers non-Baireness to the Wijsman hyperspace.

Cao proved that Baireness of $X^{\aleph_0}$ implies Baireness of the
Wijsman hyperspace for every compatible metric \cite{Cao2010}.  Thus the
countable-power hypothesis in Cao's sufficient condition cannot be replaced
by the Baireness of all finite powers.  Indeed, the space constructed here
satisfies
$$
Y^n\text{ is Baire for every }n<\aleph_0,
\qquad
Y^{\aleph_0}\text{ is not Baire}. 
$$
Our second theorem identifies a large region in which Baireness of the
base space is equivalent to the Baireness of the Wijsman hyperspace.  

The proof produces a fixed-metric obstruction: a closed meagre
$\sigma$-$d$-totally bounded subset of $X$ cannot have nonempty interior when
the Wijsman hyperspace is Baire.  The equivalence holds, in
particular, for locally $\sigma$-$d$-totally bounded spaces, $\sigma$-compact
metrizable spaces for every compatible metric, and countable metric spaces.

We finally determine the core of the counterexample exactly.  A direct
analysis of the Cao--Tomita space shows that none of its nonempty open
subsets is separable.  Consequently, for every compatible metric $e$ on $Y$,
the core $S_e$ is precisely the discrete auxiliary component $D$.  For the
metric $d$ used in the construction,
$$
S_d=D=\overline B_d(q,1),
$$
and $d|_{S_d}$ is complete, whereas $(\CL(Y),\tau_W(d))$ is not Baire.
This shows that the converse of the local core theorem fails even for a
clopen completely metrizable core, and that the density assumption in the
dense-core characterization cannot be omitted.

\section{Preliminaries}

For $U\subseteq X$, put
$$
U^-:=\{A\in\CL(X):A\cap U\neq\varnothing\},
\qquad
U^+:=\{A\in\CL(X):A\subseteq U\}.
$$
When the ambient space needs to be indicated explicitly, we write
$$
U^-_Z:=\{A\in\CL(Z):A\cap U\neq\varnothing\},
\qquad
U^+_Z:=\{A\in\CL(Z):A\subseteq U\},
\qquad U\subseteq Z.
$$
The Vietoris topology $\tau_V$ is generated by the sets $U^-$ and $U^+$,
where $U$ ranges over the open subsets of $X$.

For $x\in X$ and $r>0$, write
$\overline B_d(x,r)=\{y\in X:d(x,y)\leq r\}$.  Let $\mathcal K_d$ be the
family of all finite unions of closed $d$-balls, with the empty union included.
The ball topology $b_d$ on $\CL(X)$ is generated by the hit sets $U^-$, for
$U$ open, and the miss sets
$$
(K^c)^+=\{A\in\CL(X):A\cap K=\varnothing\},\qquad K\in\mathcal K_d.
$$
Restricting the definition to closed balls that are proper subsets of $X$ gives
the same topology: a nonproper closed ball is $X$, and its miss set is empty.

Recall that $X$ is \emph{almost locally separable} if the points possessing a
separable neighborhood form a dense subset of $X$.  We use two theorems of
Zsilinszky.

\begin{theorem}\label{thm:zsilinszky}
Let $(X,d)$ be a metric space.
\begin{enumerate}
\item $(\CL(X),\tau_W(d))$ is Baire if and only if $(\CL(X),b_d)$ is Baire.
\item If $X$ is almost locally separable, then $X$ is Baire if and only if
      $(\CL(X),\tau_W(e))$ is Baire for every compatible metric $e$ on $X$.
\end{enumerate}
\end{theorem}

These are Theorems~1.1 and~2.3, respectively, of \cite{Zsilinszky2007}.

\begin{theorem}\label{thm:seed}
There is a metrizable space $X$ such that $X^n$ is Baire for every positive
integer $n$, whereas $(\CL(X),\tau_V)$ is not Baire.
\end{theorem}

This is the space constructed in Example~3.3 of \cite{CaoTomita2007};
we use that specific realization below.

\section{A Vietoris-to-ball construction}\label{sec:marker}

The next lemma shows that, on a suitable open subspace, every upper Vietoris condition can be represented by a ball-topology miss condition.

\begin{lemma}\label{lem:marker}
Let $(X,\rho)$ be a nonempty metric space.  There exist a metric space $(Y,d)$
containing $X$ as a clopen subspace such that, under the natural identification
$\CL(X)\subseteq\CL(Y)$, the subspace $\CL(X)$ is $b_d$-open and
$$
(\CL(X),b_d|_{\CL(X)})\cong(\CL(X),\V).
$$
Moreover, $Y$ is topologically the disjoint sum of $X$ and a discrete space.
\end{lemma}

\begin{proof}
Replace $\rho$ by the equivalent metric $\rho'=\min\{\rho,1\}$, so that $\rho'\leq 1$.  For each $F\in\CL(X)$ introduce a point $p_F$, and introduce one
further point $q$, choosing all points in $\{q\}\cup\{p_F:F\in\CL(X)\}$
pairwise distinct.  Put
$$
D=\{q\}\cup\{p_F:F\in\CL(X)\},
\qquad
Y=X\mathbin{\sqcup}D.
$$
Define $d$ by
$$
\begin{aligned}
d(x,y)&=\rho'(x,y) &&(x,y\in X),\\
d(p_F,x)&=1+\rho'(x,F) &&(x\in X),\\
d(p_F,p_G)&=2 &&(F\neq G),\\
d(q,x)&=2 &&(x\in X),\\
d(q,p_F)&=1 &&(F\in\CL(X)).
\end{aligned}
$$
Together with symmetry and $d(u,u)=0$, these formulas define a metric.  For a
triangle involving $p_F,x,y$, use
$$
|\rho'(x,F)-\rho'(y,F)|\leq\rho'(x,y).
$$
For distinct auxiliary points $p_F,p_G$ and $x\in X$,
$$
d(p_F,p_G)=2\leq(1+\rho'(x,F))+(1+\rho'(x,G)).
$$
The only nontrivial inequality involving $q,p_F,x$ is
$$
d(q,x)=2\leq d(q,p_F)+d(p_F,x)=2+\rho'(x,F),
$$
while $d(p_F,x)\leq2<3=d(p_F,q)+d(q,x)$.  For distinct auxiliary points $p_F,p_G$,
$d(p_F,p_G)=2=d(p_F,q)+d(q,p_G)$.  These inequalities exhaust all triangle
types, so $d$ is indeed a metric.

The distance between $X$ and $D$ is at least $1$, and every point of $D$ is
isolated.  Hence $X$ is clopen and $Y$ is topologically
$X\mathbin{\sqcup}D$.

The construction gives
$$
\overline B_d(q,1)=D
$$
and, for every $F\in\CL(X)$,
\begin{equation}\label{eq:markerball}
\overline B_d(p_F,1)=\{p_F,q\}\cup F.
\end{equation}
Indeed, $d(p_F,x)=1$ precisely when $\rho'(x,F)=0$, which is equivalent to
$x\in F$ because $F$ is closed.

Since $X$ is clopen in $Y$, every member of $\CL(X)$ is also closed in $Y$,
so we regard $\CL(X)$ as the subspace
$$
\{A\in\CL(Y):A\cap(Y\setminus X)=\varnothing\}.
$$
Moreover,
$$
\CL(X)=(\overline B_d(q,1)^c)^+_Y,
$$
so $\CL(X)$ is open in the ball topology on $\CL(Y)$.

Using the ambient-space notation introduced above, if $O$ is open in $Y$, then
$$
O^-_Y\cap\CL(X)=(O\cap X)^-_X.
$$
If $B$ is a finite union of closed $d$-balls, then $B\cap X$ is closed in $X$,
and hence
$$
(B^c)^+_Y\cap\CL(X)
=\{A\in\CL(X):A\cap(B\cap X)=\varnothing\}
=(X\setminus(B\cap X))^+_X
$$
is Vietoris-open.  Hence $b_d|_{\CL(X)}\subseteq\V$.

Conversely, every Vietoris hit set is the trace of the corresponding hit set
in $\CL(Y)$.  Let $U\subseteq X$ be open and put $F=X\setminus U$.  If
$F\neq\varnothing$, identity \eqref{eq:markerball} gives
$$
U^+_X=\CL(X)\cap(\overline B_d(p_F,1)^c)^+_Y.
$$
The case $F=\varnothing$ gives all of $\CL(X)$.  Thus
$\V\subseteq b_d|_{\CL(X)}$, and equality follows.

If only proper closed balls are permitted, the realization still works.
Indeed, $\overline B_d(q,1)=D$ is proper because $X\neq\varnothing$.  The two
trivial upper Vietoris sets, obtained from $U=\varnothing$ and $U=X$, are
respectively $\varnothing$ and $\CL(X)$ and require no ball condition.  If
$\varnothing\subsetneq U\subsetneq X$, then $F=X\setminus U$ is nonempty and
$\overline B_d(p_F,1)$ is proper because it misses every point of the nonempty
set $U$.
\end{proof}

\section{The counterexample}\label{sec:counterexample}

\begin{theorem}\label{thm:main}
There exist a metrizable Baire space $Y$ and a compatible metric $d$ such that
$(\CL(Y),\W(d))$ is not Baire.  Moreover, $Y^n$ is Baire for every positive
integer $n$.
\end{theorem}

\begin{proof}
Let $X$ be the specific metrizable space constructed in
\cite[Example~3.3]{CaoTomita2007}, as recorded in
Theorem~\ref{thm:seed}, and choose any compatible metric $\rho$ on $X$.  Apply Lemma~\ref{lem:marker}, obtaining
$(Y,d)$.  Under the natural identification $\CL(X)\subseteq\CL(Y)$,
Lemma~\ref{lem:marker} gives
$$
(\CL(X),b_d|_{\CL(X)})\cong(\CL(X),\V),
$$
and $\CL(X)$ is $b_d$-open in $\CL(Y)$.  By hypothesis, this open subspace is not
Baire, so $(\CL(Y),b_d)$ cannot be Baire.  By
Theorem~\ref{thm:zsilinszky}(1), $(\CL(Y),\W(d))$ is not Baire.

Topologically, $Y=X\mathbin{\sqcup}D$, with $D$ discrete.  Fix $n\geq1$.  For
$S\subseteq\{1,\ldots,n\}$, let
$$
P_S=\{y\in Y^n:y_i\in X\text{ if and only if }i\in S\}.
$$
The sets $P_S$ form a finite clopen partition of $Y^n$.  Since
$D^{n-|S|}$ is discrete, $P_S$ is the topological sum of the clopen slices
$X^{|S|}\times\{\xi\}$, indexed by $\xi\in D^{n-|S|}$, and each slice is
homeomorphic to $X^{|S|}$.  Equivalently,
$$
P_S\cong X^{|S|}\times D^{n-|S|}
\cong\coprod_{\xi\in D^{n-|S|}}X^{|S|}.
$$
For $1\leq k\leq n$, the space $X^k$ is Baire by
Theorem~\ref{thm:seed}, while $X^0$ is a singleton.  An arbitrary topological
sum of Baire spaces is Baire, and therefore $Y^n$ is Baire.
\end{proof}

\begin{corollary}\label{cor:cao}
Questions~4.4 and~4.5 recorded in Cao's survey have negative answers:
\begin{enumerate}
\item Baireness of every finite power $X^n$ does not imply Baireness of the
      Wijsman hyperspace.
\item A Baire metrizable space can have a non-Baire Wijsman hyperspace for a
      compatible metric.
\end{enumerate}
\end{corollary}

\begin{corollary}\label{cor:omega-power}
For the space $Y$ in Theorem~\ref{thm:main}, every finite power $Y^n$ is
Baire, but the countable power $Y^{\aleph_0}$ is not Baire.
\end{corollary}

\begin{proof}
By \cite{Cao2010}, if $Z^{\aleph_0}$ is Baire, then
$(\CL(Z),\tau_W(e))$ is Baire for every compatible metric $e$ on $Z$.
Applying the contrapositive with $Z=Y$ and with the compatible metric $e=d$
from Theorem~\ref{thm:main} shows that $Y^{\aleph_0}$ is not Baire.
\end{proof}

\section{The locally \texorpdfstring{$\sigma$-$d$}{sigma-d}-totally bounded core}\label{sec:core}

We next establish a complementary metric obstruction expressed directly in
the original metric.

\begin{definition}
A subset $M\subseteq X$ is \emph{$\sigma$-$d$-totally bounded}, or
\emph{$\sigma$-$d$-precompact}, if
$$
M=\bigcup_{n<\aleph_0}M_n
$$
for $d$-totally bounded sets $M_n$.  Define
$$
S_d=\bigcup\{U\subseteq X:U\text{ is open and }\sigma\text{-}d
\text{-totally bounded}\}.
$$
\end{definition}

The set $S_d$ is open and is precisely the set of points having a
$\sigma$-$d$-totally bounded neighborhood.  Moreover, every
$\sigma$-$d$-totally bounded metric subspace is separable: a totally bounded
metric space is separable, and a countable union of separable subspaces is
separable.  Thus every point of $S_d$ has an open separable neighborhood, and
$S_d$ is contained in the set of points of local separability.

\begin{lemma}\label{lem:decomposition}
Let $C$ be a closed, meagre, $\sigma$-$d$-totally bounded subset of $X$.
Then there are increasing closed nowhere dense $d$-totally bounded sets $C_n$
such that
$$
C=\bigcup_{n<\aleph_0}C_n.
$$
\end{lemma}

\begin{proof}
Choose closed nowhere dense sets $N_i$ covering $C$ and totally bounded sets
$T_j$ covering $C$.  For every pair $(i,j)$, the set
$$
D_{i,j}=\overline{C\cap N_i\cap T_j}^{\,X}
$$
is closed, nowhere dense, and totally bounded: it is contained in the closed
nowhere dense set $C\cap N_i$, and the closure of a totally bounded set remains
totally bounded.  The sets $D_{i,j}$ cover $C$.  Enumerate the pairs $(i,j)\in\omega^2$,
putting a pair with $D_{i,j}\neq\varnothing$ first when $C\neq\varnothing$,
and take successive finite unions of the corresponding $D_{i,j}$.  This gives
the required sequence $(C_n)_{n<\aleph_0}$.
\end{proof}

\begin{proposition}\label{prop:sigma-obstruction}
Assume that $(\CL(X),\W(d))$ is Baire.  If $C\subseteq X$ is closed, meagre
in $X$, and $\sigma$-$d$-totally bounded, then
$$
\operatorname{int} C=\varnothing.
$$
\end{proposition}

\begin{proof}
Suppose, towards a contradiction, that $\operatorname{int} C\neq\varnothing$.
Apply Lemma~\ref{lem:decomposition} and write
$$
C=\bigcup_{n<\aleph_0}C_n,
$$
where the $C_n$ are nonempty, increasing, closed, nowhere dense, and totally
bounded.
For every $n$ and $k\geq1$, choose a finite set $F_{n,k}\subseteq C_n$ such that
$$
C_n\subseteq\bigcup_{f\in F_{n,k}}B_d\left(f,\frac1{2k}\right),
$$
and put
$$
K_{n,k}=\bigcup_{f\in F_{n,k}}\overline B_d\left(f,\frac1k\right),
\qquad
G_n=\bigcup_{k\geq1}(K_{n,k}^c)^+.
$$
Each $G_n$ is ball-open.  We claim that it is dense.  Let
$$
\mathcal U=(K^c)^+\cap\bigcap_{i=1}^m U_i^-,\qquad m\geq0,
$$
be a nonempty basic ball-open set, where each $U_i\subseteq X$ is open and
$K$ is a finite union of closed $d$-balls.  Since every $A\in(K^c)^+$ is contained
in $X\setminus K$, replacing each $U_i$ by $U_i\setminus K$ does not change
$\mathcal U$.  We may therefore assume that $U_i\subseteq X\setminus K$ for
every $i$.  If $m=0$, add the redundant hit condition $U_1=X\setminus K$, so
that $m\geq1$.  Since $C_n$ is nowhere dense, choose
$$
v_i\in U_i\setminus C_n
$$
for each $i$.  Because $C_n$ is closed and there are only finitely many $v_i$,
$$
\delta=\min_i d(v_i,C_n)>0.
$$
Choose $k$ with $\frac 1k<\delta$.  Then the finite closed set
$A=\{v_1,\ldots,v_m\}$ avoids $K_{n,k}$ and belongs to
$\mathcal U\cap G_n$.  Thus $G_n$ is dense in the ball topology.

On the other hand,
$$
(\operatorname{int} C)^-\cap\bigcap_{n<\aleph_0}G_n=\varnothing.
$$
Indeed, if $A$ meets $\operatorname{int} C$, choose
$a\in A\cap\operatorname{int} C$ and then
$n$ with $a\in C_n$.
For every $k$, the covering property gives $a\in K_{n,k}$, so
$A\notin(K_{n,k}^c)^+$ and hence $A\notin G_n$.  We have produced dense open
sets whose intersection misses the nonempty $b_d$-open set
$(\operatorname{int}C)^-$ and is therefore not dense.  Thus
$(\CL(X),b_d)$ is not Baire,
and Theorem~\ref{thm:zsilinszky}(1) implies that $(\CL(X),\W(d))$ is not
Baire.  We obtain a contradiction, and the claim follows.
\end{proof}

\begin{remark}
The preceding proof is a fixed-metric localization of the obstruction used in
the proof of Zsilinszky's Theorem~2.3 \cite{Zsilinszky2007}.  There the metric
is replaced on a suitable separable closed subspace by a totally bounded
compatible metric.  Here the $\sigma$-$d$-total boundedness already present in
the original metric allows the argument to be carried out without changing
$d$.
\end{remark}

\begin{theorem}\label{thm:kernel}
If $(\CL(X),\W(d))$ is Baire, then $S_d$ is Baire.
\end{theorem}

\begin{proof}
Suppose that $S_d$ is not Baire.  There is a nonempty set
$O\subseteq S_d$ which is open in $S_d$ and meagre in $S_d$.  Since $S_d$ is
open in $X$, the set $O$ is also open in $X$.  Choose
$x\in O$ and an open $\sigma$-$d$-totally bounded neighborhood $U$ of $x$.
By metric regularity, there is a nonempty open set $V$ such that
$$
C:=\overline V^{\,X}\subseteq O\cap U.
$$
The set $C$ is closed, has nonempty interior, and is
$\sigma$-$d$-totally bounded, since that property is hereditary and $C\subseteq
U$.  It is also meagre in $X$: since $S_d$ is open in $X$, every subset that
is meagre in $S_d$ is meagre in $X$.  Hence $C\subseteq O$ is meagre in $X$.
This contradicts Proposition~\ref{prop:sigma-obstruction}.
\end{proof}

\begin{corollary}\label{cor:dense-kernel}
If $S_d$ is dense in $X$, then the following are equivalent:
\begin{enumerate}
\item $(\CL(X),\W(d))$ is Baire;
\item $X$ is Baire;
\item $S_d$ is Baire.
\end{enumerate}
\end{corollary}

\begin{proof}
Theorem~\ref{thm:kernel} gives (1)$\Rightarrow$(3).  The open dense Baire
subspace $S_d$ forces $X$ to be Baire, so (3)$\Rightarrow$(2).  Finally, if
$X$ is Baire, the open subspace $S_d$ is Baire.  Since $S_d$ is dense and every
point of $S_d$ has a separable neighborhood, $X$ is almost locally separable.
Theorem~\ref{thm:zsilinszky}(2) gives (2)$\Rightarrow$(1).
\end{proof}

We now return to the counterexample of Section~\ref{sec:counterexample}.
Corollary~\ref{cor:marker-core} identifies its locally
$\sigma$-$d$-totally bounded core exactly.  The only additional ingredient is the following
local property of the specific Cao--Tomita space.

\begin{lemma}\label{lem:cao-tomita-nonseparable}
Let $X$ be the metric space constructed in
\cite[Example~3.3]{CaoTomita2007}.  Every nonempty open subset of $X$
contains $\mathfrak c^+$ pairwise disjoint nonempty open subsets.  In
particular, $X$ has no nonempty separable open subset.
\end{lemma}

\begin{proof}
Put $\kappa=\mathfrak c^+$, where $\mathfrak c=2^{\aleph_0}$, and, for a
cardinal $\lambda$, put $J_\lambda=\lambda^\omega$, endowed with the
product of the discrete topology on $\lambda$.  If
$s\in\lambda^{<\omega}$, write
$$
[s]=\{g\in J_\lambda:s\subseteq g\}.
$$
Recall the construction from \cite[Example~3.3]{CaoTomita2007}.  There is a
family $\{A_z:z\in J_\omega\}$ of pairwise disjoint stationary subsets of
$$
C^\kappa_\omega=\{\alpha<\kappa:\operatorname{cf}(\alpha)=\omega\}.
$$
For $y\in J_\omega$, define
$$
I_y=\left\{z\in J_\omega:
\max\{n<\omega:y(n)=z(n)\}
\leq
\max\{n<\omega:z(n)\geq n\}<\omega
\right\}
$$
and $C_y=\bigcup_{z\in I_y}A_z$.  For $f\in J_\kappa$, put
$$
f^*=\sup_{n<\omega}(f(n)+1).
$$
Then
$$
X=\{(y,f)\in J_\omega\times J_\kappa:f^*\in C_y\}.
$$

We next show that $X$ meets every basic open rectangle in
$J_\omega\times J_\kappa$.  Let $s\in\omega^{<\omega}$ and
$t\in\kappa^{<\omega}$.  Choose $y\in[s]$ and then $z\in I_y$.
Since $A_z$ is stationary in the regular uncountable cardinal $\kappa$, it
is unbounded.  Choose $\alpha\in A_z$ larger than every value occurring in
$t$.  As $\operatorname{cf}(\alpha)=\omega$, there is a sequence cofinal in
$\alpha$ all of whose terms are larger than every value occurring in $t$.
Appending such a sequence to $t$ gives $f\in[t]$ with $f^*=\alpha$.
Since $\alpha\in A_z\subseteq C_y$, we have
$$
(y,f)\in X\cap([s]\times[t]).
$$
Thus $X$ is dense in $J_\omega\times J_\kappa$.

Let $U$ be a nonempty open subset of $X$.  Choose
$s\in\omega^{<\omega}$ and $t\in\kappa^{<\omega}$ such that
$$
\varnothing\neq X\cap([s]\times[t])\subseteq U.
$$
For each $\xi<\kappa$, put
$$
U_\xi=X\cap\bigl([s]\times[t^\frown\langle\xi\rangle]\bigr).
$$
The sets $U_\xi$ are open in $X$, pairwise disjoint, and contained in $U$;
they are nonempty by the density just proved.  Hence $U$ contains $\kappa$
pairwise disjoint nonempty open subsets.  A separable metric space has no
uncountable family of pairwise disjoint nonempty open subsets, so $U$ is not
separable.
\end{proof}

\begin{corollary}\label{cor:marker-core}
Let $X$ be the Cao--Tomita space from Lemma~\ref{lem:cao-tomita-nonseparable},
and let $(Y,d)$ be the auxiliary-point space constructed from $X$ in the proof of
Theorem~\ref{thm:main}, so that $Y=X\mathbin{\sqcup}D$.  For every compatible
metric $e$ on $Y$, the locally $\sigma$-$e$-totally bounded core satisfies
$$
S_e=D.
$$
In particular,
$$
S_d=D=\overline B_d(q,1).
$$
Thus $S_d$ is clopen and discrete, $d|_{S_d}$ is complete, and $S_d$ is
Baire, whereas $(\CL(Y),\W(d))$ is not Baire.
\end{corollary}

\begin{proof}
Every point of $D$ is isolated.  Hence its singleton is open and
$e$-totally bounded, and therefore $D\subseteq S_e$.

Suppose that $x\in X\cap S_e$.  There is an open
$\sigma$-$e$-totally bounded set $U\subseteq Y$ containing $x$.  Then
$U\cap X$ is a nonempty open $\sigma$-$(e|_X)$-totally bounded subset of
$X$, and is therefore separable.  This contradicts
Lemma~\ref{lem:cao-tomita-nonseparable}.  Thus $S_e\cap X=\varnothing$, and
$S_e=D$.

For the metric $d$ of Lemma~\ref{lem:marker}, the identity
$D=\overline B_d(q,1)$ was established in its proof.  Distinct points of
$D$ are at $d$-distance at least $1$, so $d|_D$ is complete.  The final
assertion follows from Theorem~\ref{thm:main}.
\end{proof}

\begin{remark}\label{rem:core-sharpness}
In the counterexample, $Y^n$ is Baire for every positive integer $n$, and
$S_d$ is clopen, discrete, and complete in its restricted metric, but
$(\CL(Y),\W(d))$ is not Baire.  Thus the converse of
Theorem~\ref{thm:kernel} fails even for a clopen completely metrizable core,
and the density assumption in Corollary~\ref{cor:dense-kernel} cannot be
omitted.  Moreover, the open subspace used in the construction is
exactly
$$
\CL(X)=(Y\setminus S_d)^+_Y.
$$
Hence the non-Baire open copy of $(\CL(X),\V)$ consists precisely of the
closed subsets of $Y$ that avoid the core.
\end{remark}

\begin{corollary}\label{cor:locally-sigma-precompact}
If $(X,d)$ is locally $\sigma$-$d$-totally bounded, then
$$
(\CL(X),\W(d))\text{ is Baire}
\quad\Longleftrightarrow\quad
X\text{ is Baire}.
$$
\end{corollary}

\begin{proof}
In this case $S_d=X$, so Corollary~\ref{cor:dense-kernel} applies.
\end{proof}

\begin{corollary}\label{cor:sigma-precompact}
If $(X,d)$ is $\sigma$-$d$-totally bounded, then
$$
(\CL(X),\W(d))\text{ is Baire}
\quad\Longleftrightarrow\quad
X\text{ is Baire}.
$$
In particular, the equivalence holds when $(X,d)$ is totally bounded or when
$X$ is countable.
\end{corollary}

\begin{proof}
Here $X$ itself is an open $\sigma$-$d$-totally bounded set, so $S_d=X$ and
Corollary~\ref{cor:dense-kernel} applies.  A countable metric space is a
countable union of singletons, each of which is totally bounded.
\end{proof}

\begin{corollary}\label{cor:sigma-compact}
If $X$ is a $\sigma$-compact metrizable space, then for every compatible metric
$d$,
$$
(\CL(X),\W(d))\text{ is Baire}
\quad\Longleftrightarrow\quad
X\text{ is Baire}.
$$
\end{corollary}

\begin{proof}
Every compact subset of a metric space is totally bounded.  Hence a
$\sigma$-compact metric space is $\sigma$-$d$-totally bounded for every
compatible metric $d$, and Corollary~\ref{cor:sigma-precompact} applies.
\end{proof}

\section{Concluding remarks}
Theorem~\ref{thm:main} settles Question~4.4 of Cao's survey and gives a
negative answer to the question posed by Zsilinszky at the 10th Prague
Topological Symposium in 2006, recorded there as Question~4.5.  The construction shows, in particular, that finite-power Baireness does not control the
Baire property of a Wijsman hyperspace.  Together with Cao's theorem for the
countable power, it gives a precise separation between finite powers and
$X^{\aleph_0}$.

The core $S_d$ provides a complementary positive result.  Whenever $S_d$ is
dense in $X$, Baireness of the Wijsman hyperspace is equivalent to Baireness
of the base space.  The counterexample shows that the density assumption is
sharp in a strong form.  For every compatible metric $e$ on the auxiliary-point space
$Y$, its core is exactly the discrete auxiliary component $D$.  For the metric
$d$ used to realize the Vietoris hyperspace, $S_d=D$ is clopen and complete
in its restricted metric, while $(\CL(Y),\tau_W(d))$ is not Baire.  In
addition, the open copy of the non-Baire Vietoris hyperspace is
$\CL(X)=(Y\setminus S_d)^+_Y$, so the obstruction is carried by closed
sets that avoid the core altogether.


\begin{thebibliography}{99}
\small

\bibitem{Cao2010}
J.~Cao,
\emph{The Baire property in hit-and-miss hypertopologies},
Topology Appl. \textbf{157} (2010), no.~8, 1325--1334.
\href{https://doi.org/10.1016/j.topol.2009.03.053}{doi:10.1016/j.topol.2009.03.053}.

\bibitem{Cao2017}
J.~Cao,
\emph{Three open problems on the Wijsman topology},
Questions Answers Gen. Topology \textbf{35} (2017), no.~1, 49--58.
\href{https://doi.org/10.82665/qagt.35.1_49}{doi:10.82665/qagt.35.1\_49}.

\bibitem{CaoGarciaFerreiraGutev2007}
J.~Cao, S.~Garc\'ia-Ferreira and V.~Gutev,
\emph{Baire spaces and Vietoris hyperspaces},
Proc. Amer. Math. Soc. \textbf{135} (2007), no.~1, 299--303.
\href{https://doi.org/10.1090/S0002-9939-06-08743-0}{doi:10.1090/S0002-9939-06-08743-0}.

\bibitem{CaoGarciaFerreiraGutev2008}
J.~Cao, S.~Garc\'ia-Ferreira and V.~Gutev,
\emph{Erratum to ``Baire spaces and Vietoris hyperspaces''},
Proc. Amer. Math. Soc. \textbf{136} (2008), no.~10, 3729--3731.
\href{https://doi.org/10.1090/S0002-9939-08-09291-5}{doi:10.1090/S0002-9939-08-09291-5}.

\bibitem{CaoTomita2007}
J.~Cao and A.~H.~Tomita,
\emph{Baire spaces, Tychonoff powers and the Vietoris topology},
Proc. Amer. Math. Soc. \textbf{135} (2007), no.~5, 1565--1573.
\href{https://doi.org/10.1090/S0002-9939-07-08855-7}{doi:10.1090/S0002-9939-07-08855-7}.

\bibitem{GarciaFerreiraGutev2009}
S.~Garc\'ia-Ferreira and V.~Gutev,
\emph{Baire property and web-adjacent spaces},
Houston J. Math. \textbf{35} (2009), no.~3, 857--875.

\bibitem{Zsilinszky2007}
L.~Zsilinszky,
\emph{On Baireness of the Wijsman hyperspace},
Boll. Unione Mat. Ital. (8) \textbf{10-B} (2007), no.~3, 1071--1079.
\href{https://eudml.org/doc/290370}{EuDML 290370}.

\end{thebibliography}
\end{document}